\documentclass[11pt]{article}
\usepackage{rotating}
\usepackage{caption}
\usepackage{amssymb,amsthm,amsmath,graphicx}
\usepackage[latin1]{inputenc}
\newtheorem{thm}{Theorem}
\newtheorem{lm}[thm]{Lemma}

\newtheorem{deff}[thm]{Definition}

\def\blfootnote{\xdef\@thefnmark{}\@footnotetext}

\date{}
\author{MARIA DE LOURDES MERLINI GIULIANI\\ \textit{Universidade
    Federal do ABC}\\ \textit{Santo André (SP),
    09210-180 Brazil}\\ 
  {GILIARD SOUZA DOS ANJOS}\footnote{Partially supported by FAPESP, Proc. 2012/21323-0}\\  \textit{Universidade Federal do ABC} \\
  \textit{Santo André (SP), 09210-180 Brazil}\\
  CHARLES J. COLBOURN\\
\textit{School of CIDSE, Arizona State University}\\
\textit{Tempe, AZ 85287, U.S.A}
}
\title{STEINER LOOPS SATISFYING THE STATEMENT OF MOUFANG'S THEOREM}

\begin{document}

\maketitle

\begin{abstract} 
Andrew Rajah posed at the Loops'11 Conference in Trest, Czech Republic, the following conjecture: 
\textsl{Is every variety of loops that satisfies Moufang's theorem contained in the variety of Moufang loops?}
This paper is motivated by that problem. 
We give a partial answer to this question and present two types of Steiner loops, one  that satisfies Moufang's theorem and another that does not, and neither is Moufang loop.
\end{abstract}

\noindent {\it Keywords}: \: Steiner loop; Moufang loop; Moufang theorem.\\
 \noindent Mathematics Subject Classification:  20N05.

\section {INTRODUCTION}

A nonempty set $L$ with a binary operation is a \it loop \rm if there exists an identity
element $1$ with $1x = x = x1$ for every $x \in L$ and both left and right multiplication
by any fixed element of $L$ permutes every element of $L$.

A loop $L$ has the \textit{inverse property} (and is an \textsl{IP loop}), if and only if there is a bijection 
$L\longrightarrow L:x\mapsto x^{-1}$ such whenever $x,y\in L$, $x^{-1}(xy)=y=(yx)x^{-1}$. 
It can be seen that \textit{IP} loops also satisfy  $(xy)^{-1}=y^{-1}x^{-1}$. 
A \textit{Steiner loop} is an IP loop of exponent $2$.
A loop $M$  is a  \it Moufang loop \rm if it satisfies any of the following equivalent identities: 
\[ \begin{array}{rcl}
x(y(xz))&=&((xy)x)z\\
y(x(zx))&=&((yx)z)x\\
(xy)(zx)&=&x((yz)x)
\end{array} \]
Such loops were introduced by Moufang \cite{Moufang}  in 1934. 
The \textit{associator} $(a,b,c)$ of  elements $a,b,c$ is the unique element of $L$ satisfying the equation: $ab.c\,=\,(a.bc)(a,b,c)$. 
\begin{thm}\label{mout}[Moufang's Theorem \cite{hala}]
Let $M$ be a Moufang loop. If  $a,b,c \in M$ such that $(a,b,c)=1$ then $a,b,c$ generate a subgroup of $M$.
\end{thm}
In view of Theorem \ref{mout}, every Moufang loop is \it diassociative\rm, that is, any two of its elements generate a group.
However, Theorem \ref{mout} was formulated for Moufang loops. 
We consider its statement for another class of loops, namely, for the variety of Steiner loops. 

\section {STEINER LOOPS AND MOUFANG'S THEOREM}

\begin{deff}
A loop $L$ \textit{satisfies Moufang's Property}, $\mathcal{MP}$, if $L$ is not Moufang loop, but it satisfies the statement of Moufang's theorem, i.e., if $a,b,c \in L$ such that $(a,b,c)=1$ then $a,b,c$ generate a subgroup of $L$.
\end{deff}

It is known that there exists only one Steiner loop $S$ of order $10$. 
We prove that this Steiner loop $S$ satisfies Moufang's Property $\mathcal{MP}$. 
Its Cayley table can be found, for example, using the GAP Library \cite{gap}, as seen below:

\begin{center}
\begin{tabular}{|c|c|c|c|c|c|c|c|c|c|c|c|c|c|c|c|c|}
\hline
  1& 2& 3& 4& 5& 6& 7& 8& 9& 10 \\
\hline
   2& 1& 4& 3& 8& 10& 9& 5& 7& 6 \\
\hline
   3& 4& 1& 2& 10& 9& 8& 7& 6& 5 \\
\hline
  4& 3& 2& 1& 9& 8& 10& 6& 5& 7 \\
\hline
   5& 8& 10& 9& 1& 7& 6& 2& 4& 3 \\
\hline
  6& 10& 9& 8& 7& 1& 5& 4& 3& 2 \\
\hline
   7& 9& 8& 10& 6& 5& 1& 3& 2& 4 \\
\hline
 8& 5& 7& 6& 2& 4& 3& 1& 10& 9 \\
\hline
   9& 7& 6& 5& 4& 3& 2& 10& 1& 8 \\
\hline
 10& 6& 5& 7& 3& 2& 4& 9& 8& 1 \\

\hline
\end{tabular}
\end{center}

For any $x,y,z \in S$, such that $x \neq y; \,\, y \neq z; \,\, z \neq x$,
 $x \neq 1,y \neq 1, z \neq 1, \,\,\,\,x.yz\,=\,xy.z$ implies that $z\,=\,xy$. 
So  $<x,y,z>\,=\, <x,y>$, and hence $x,y,z$ generate a group.

A \textit{Steiner triple system} $(Q,{\mathcal B})$, or \textit{STS(n)}, is  a non-empty set $Q$ with $n$ elements and a set $\mathcal B$ of unordered triples $\{a,b,c\}$ such that
\begin{itemize}
\item[(i)] $a,b,c$ are distinct elements of $Q$;
\item[(ii)] when $a,b \in Q$ and $a \neq b$, there exists a unique triple $\{a,b,c\} \in {\mathcal B}$.
\end{itemize}
A Steiner triple system $(Q,{\mathcal B})$ with $|Q|=n$ elements exists if and only if $n\geq 1$ and $n\equiv 1 \,\,or \,\,3 \,\,\,(\textrm{mod}\,\, 6)$ \cite{lint}.
Because there is a one-to-one correspondence between the variety of Steiner triple systems and the 
variety of all Steiner Loops \cite{gp},  Steiner loops have order $m\equiv 2 \,\,or \,\,4 \,\,\,(\textrm{mod}\,\, 6)$. 
This underlies the study of Steiner triple systems from an algebraic point of view as in \cite{hala}, \cite{ss} and \cite{ss1}.

We use the following standard construction of Steiner triple systems \cite{lint}, sometimes called the {\sl Bose construction}.
Let $n= 2t +1$ and define  $Q:= {\mathbb{Z}}_n \times {\mathbb{Z}}_3$. A Steiner triple system $(Q,{\mathcal B})$ can be formed with ${\mathcal B}$ consisting of the following triples 
\[ \begin{array}{lll}
\{(x,0),(x,1),(x,2)\} &  \text {where}  & x \in {\mathbb{Z}}_n,  \mbox{ and}   \\
 \{(x,i),(y,i),(\frac{x+y}{2}, i+1)\}   &  \text{where}   & x \neq y; x,y \in {\mathbb{Z}}_n, i \in {\mathbb Z}_3  
\end{array} \]

The corresponding Steiner loops can be defined directly.
Let $S= Q \cup \{1\}$. 
Define a binary operation $*$ with identity element $1$ as follows:
\[ \begin{array}{rcll}
(x,i) * (x,j)&=& ( x, k)& i \neq j, \, i \neq k, \, j \neq k \\
(x,i) * (y,i)&=& ( \frac{x+y}{2}, i +1)&x \neq y \\
(x,i) * (y,i+1)&=& ( 2y - x,i)&x \neq y \\
(x,i) * (y,i-1)&=& ( 2x - y,i-1)&x \neq y \\
(x,i) * (x,i)&=& 1 & 
\end{array} \] 
Then $(S,*)$ is commutative loop.

Analyzing Steiner loops from the Bose construction, there are two types: one that satisfies $\mathcal{MP}$, and another that does not.
Using computer calculations and the {\tt Loops} package in GAP \cite{gap}, first we studied the Steiner loops of order $k$ with $k \in M_1$ where 
\begin{multline*}
M_1=\{ 16, 28,34,40,46,52,58,79,76,82,\\
88,94,100,112,118,124,130,136,142,154\}
\end{multline*}
 from the Bose construction. 
Each of these Steiner loops satisfies $\mathcal{MP}$.
However, none of the Steiner loops of order $k$ from the Bose construction with $k  \in \{22,64,106, 148\}$ satisfies $\mathcal{MP}$.  The explanation for this follows. 

\begin {thm}
Let $S$ be a Steiner loop from the Bose construction. Then  $S$ has the property $\mathcal{MP}$ if and only if $7$ is an invertible element in ${\mathbb{Z}}_n$.
\end{thm}
\begin{proof}
Suppose $S$ has property $\mathcal{MP}$. If $7$ is not invertible in ${\mathbb{Z}}_n $, then exists an element $a \in {\mathbb{Z}}_n $, $a \neq 0$    such that $7a=0$. 
Hence $8a=a$.  Because $n$ is odd, $2a=a/4$.
The associator $((0,1),(0,0),(a,0))\, =\,1 $ while $((0,1),(a,0),(0,0)) \neq 1$, thus the elements $(0,1),(0,0),(a,0) $ associate in some order, but not in every order, a contradiction. 

Now, suppose that $7$ is invertible in ${\mathbb{Z}}_n $. We consider all possible triples of elements of $S$. Our strategy is to show that if the associator $(a,b,c) = 1$, then $a,b,c$ are in the same  triple. 
There are $25$ generic triple elements of $S$;  here $x,y,z \in \mathbb{Z}_n$ are distinct  and $i,j,k \in \mathbb{Z}_3$ are distinct:

\hspace{-0.6cm}$\{(x,i),(x,i),(x,i)\}, \{(x,i),(x,i),(x,j)\},\{(x,i),(x,i),(y,i)\},\{(x,i),(x,i),(y,j)\}$,
\\
$\{(x,i),(x,j),(x,i)\}, \{(x,i),(x,j),(x,j)\},\{(x,i),(x,j),(y,i)\},\{(x,i),(x,j),(y,j)\}$,
\\
$\{(x,i),(x,j),(x,k)\},\{(x,i),(x,j),(y,k)\}$, $\{(x,i),(y,i),(x,i)\}, \{(x,i),(y,i),(x,j)\}$,
\\
$\{(x,i),(y,i),(y,i)\},\{(x,i),(y,i),(y,j)\}, \{( x,i),(y,i),(z,i)\}, \{(x,i),(y,i),(z,j)\}$,
\\
$\{(x,i),(y,j),(x,i)\},\{(x,i),(y,j),(x,j)\},\{(x,i),(y,j),(y,i)\}, \{(x,i),(y,j),(y,j)\}$,
\\
$\{(x,i),(y,j),(z,i)\},\{(x,i),(y,j),(z,j)\},\{(x,i),(y,j),(x,k)\}, \{(x,i),(y,j),(y,k)\}$,
\\
$\{(x,i),(y,j),(z,k)\}$.

When we consider $j \neq i,\,\,j \neq k,\,\, k \neq i$, we assume that $j=i+1$ and $k=i - 1$ or $j=i- 1$ and $k=i + 1$. 
We identify $59$ different sets of triples of elements. There are $37$ triples in which the associator is different from $1$, as follows:

\hspace{-0.6cm}$\{(x,i),(x,i+1),(y,i)\}, \{(x,i),(x,i-1),(y,i)\},\{(x,i),(x,i+1),(y,i+1)\}$,
\\
$\{(x,i),(y,i),(x,i+1)\},\{(x,i),(y,i),(x,i-1)\},\{(x,i),(y,i),(y,i-1)\}$,
\\
$\{(x,i),(y,i),(z,i), $ where $z \not = \frac{x+y}{2} $ and $ x\not = \frac{y+z}{2}\}$,
\\
$\{(x,i),(y,i),(z,i), $ where $z \not = \frac{x+y}{2} $ and $ x = \frac{y+z}{2}\}$,
\\
$\{(x,i),(y,i),(z,i), $ where $z = \frac{x+y}{2} $ and $ x\not = \frac{y+z}{2}\}$,
\\
$\{(x,i),(y,i),(z,i), $ where $z  = \frac{x+y}{2} $ and $ x = \frac{y+z}{2}\}$,
\\
$\{(x,i),(y,i),(z,i+1), $ where $z \not = \frac{x+y}{2} \},\{(x,i),(y,i+1),(x,i+1)\},$
\\
$\{(x,i),(y,i),(z,i-1), $ where $z \not = \frac{x+y}{2} $ and $ x\not = 2y-z\}$,
\\
$\{(x,i),(y,i),(z,i-1), $ where $z  = \frac{x+y}{2} $ and $ x\not = 2y-z\}$,
\\
$\{(x,i),(y,i),(z,i-1), $ where $z  = \frac{x+y}{2} $ and $ x = 2y-z\},$
\\
$\{(x,i),(y,i-1),(x,i-1)\},\{(x,i),(y,i+1),(y,i)\},\{(x,i),(y,i-1),(y,i)\}$,
\\
$\{(x,i),(y,i+1),(z,i), $ where $z \not = 2y-x\}$,
\\
$\{(x,i),(y,i-1),(z,i), $ where $z  \not = 2x-y $ and $ x \not= 2z-y\},$
\\
$\{(x,i),(y,i-1),(z,i), $ where $z  \not = 2x-y $ and $ x = 2z-y\},$
\\
$\{(x,i),(y,i-1),(z,i), $ where $z   = 2x-y $ and $ x \not= 2z-y\},$
\\
$\{(x,i),(y,i-1),(z,i), $ where $z   = 2x-y $ and $ x = 2z-y\},$
\\
$\{(x,i),(y,i+1),(z,i+1), $ where $z  \not = 2y-x $ and $ x \not= \frac{y+z}{2}\},$
\\
$\{(x,i),(y,i+1),(z,i+1), $ where $z  \not = 2y-x $ and $ x = \frac{y+z}{2}\},$
\\
$\{(x,i),(y,i+1),(z,i+1), $ where $z   = 2y-x $ and $ x = \frac{y+z}{2}\},$
\\
$\{(x,i),(y,i-1),(z,i-1), $ where $z  \not = 2x-y \},\{(x,i),(y,i+1),(x,i-1)\},$
\\
$\{(x,i),(y,i-1),(x,i+1)\},\{(x,i),(y,i+1),(y,i-1)\},$
\\
$\{(x,i),(y,i+1),(z,i-1), $ where $z  \not = 2y-x $ and $ x \not= 2z-y\},$
\\
$\{(x,i),(y,i+1),(z,i-1), $ where $z   = 2y-x $ and $ x \not= 2z-y\},$
\\
$\{(x,i),(y,i+1),(z,i-1), $ where $z  = 2y-x $ and $ x = 2z-y\},$
\\
$\{(x,i),(y,i-1),(z,i+1), $ where $z  \not = 2x-y $ and $ x \not= 2y-z\},$
\\
$\{(x,i),(y,i-1),(z,i+1), $ where $z  \not = 2x-y $ and $ x = 2y-z\},$
\\
$\{(x,i),(y,i-1),(z,i+1), $ where $z   = 2x-y $ and $ x = 2y-z\},$
\\
$\{(x,i),(x,i-1),(y,i+1)\}$.

Next, there are  $14$ triples for which the associator is $1$ and they are in the same triple of the STS; consequently, they are in a Klein group (and so generate a subgroup). 

\hspace{-0.6cm}$\{(x,i),(x,i),(x,i)\}, \{(x,i),(x,i),(x,j)\},\{(x,i),(x,i),(y,i)\},\{(x,i),(x,i),(y,j)\}$,
\\
$\{(x,i),(x,j),(x,i)\}, \{(x,i),(x,j),(x,j)\},\{(x,i),(y,i),(x,i)\},\{(x,i),(y,i),(y,i)\}$,
\\
$\{(x,i),(y,i),(z,i+1), $ where $z = \frac{x + y}{2}\},\{(x,i),(y,j),(x,i)\},\{(x,i),(y,j),(y,j)\}$,
\\
$\{(x,i),(y,i+1),(z,i), $ where $z = 2y-x\},
\\
\{(x,i),(y,i-1),(z,i-1), $ where $z = 2x-y\},\{(x,i),(x,j),(x,k)\}$

There remain $8$ cases to consider:

\hspace{-0.6cm}$\{(x,i),(x,i-1),(y,i-1)\}, \{(x,i),(y,i),(y,i+1)\},$
\\
$\{(x,i),(y,i-1),(y,i+1)\},\{(x,i),(x,i+1),(y,i-1)\}$,
\\
$\{(x,i),(y,i),(z,i-1), $ where $ x = 2y - z, z \not =\frac{x + y}{2} \},$
\\
$\{(x,i),(y,i+1),(z,i+1) $ where $ z = 2y - x, x \not =\frac{y + z}{2}\},
\\
\{(x,i),(y,i+1),(z,i-1) $ where $ z \not = 2y - x, x = 2z - y\}$,
\\
$\{(x,i),(y,i-1),(z,i+1) $ where $ z = 2x - y, x \not = 2y - z\}$

Each has associator different from $1$ because $7$ is invertible in $\mathbb{Z}_n$.
Take for instance the triple $\{(x,i),(x,i+1),(y,i-1)\}$ with $x \neq y$ of the STS.
Now $(x,i)*((x,i+1)*(y,i-1)) = (4y - 3x,i)$ and
$((x,i)*(x,i+1))*(y,i-1) = (\frac{x+y}{2},i)$.
The associator $((x,i),(x,i+1),(y,i-1))$ is $1$ if and only if $7x = 7y$. 
Because  $7$ is  invertible in $\mathbb{Z}_n$, we obtain $x = y$, a contradiction.
\end{proof}

\section{BEYOND STEINER LOOPS}

We have seen that certain Steiner loops from the Bose construction provide examples of loops satisfying $\mathcal{MP}$.  
Further examples can be obtained by the direct product of loops, the proof of which is straightforward:

\begin{lm}
Let $S$ and $M$ be loops that satisfy Moufang's theorem.  Then $S \times M$ satisfies Moufang's theorem, and $S \times M$ satisfies $\mathcal{MP}$ if one or both of $S$ and $M$ satisfy $\mathcal{MP}$.
\end{lm}

Taking $S$ to satisfy $\mathcal{MP}$ and $M$ to be a group or a Moufang loop provides numerous examples of loops that satisfy $\mathcal{MP}$ but are neither Steiner nor Moufang loops.
A characterization of loops that satisfy Moufang's theorem must therefore consider loops beyond the varieties examined here.

\end{document}